\documentclass[11pt]{amsart}

\usepackage{amsmath, amssymb, amsthm, mathtools}
\usepackage{mathrsfs}
\usepackage{enumerate}
\usepackage{hyperref}
\usepackage{geometry}
\usepackage{graphicx}
\theoremstyle{plain}
\newtheorem{theorem}{Theorem}[section]
\newtheorem{corollary}[theorem]{Corollary}

\newtheorem{lemma}[theorem]{Lemma}

\theoremstyle{definition}

\theoremstyle{remark}
\newtheorem{remark}[theorem]{Remark}

\title[H\"older regularity for an elliptic PDE]{Improved H\"older regularity for elliptic equations of non-divergence type in the plane}
\author{Zhiqiang Hou}
\address{Zhiqiang Hou,  Department of Mathematics,
Shantou University,
Shantou 515000,
China.}
\email{25zqhou@stu.edu.cn}
\author{Jian-Feng Zhu}
\address{Department of Mathematics, Shantou University, Shantou, Guangdong, China}
\email{flandy@stu.edu.cn}
\date{\today}

\begin{document}

\begin{abstract}
In this paper, we obtain an improved H\"older regularity for quasiregular gradient mappings which was studied in \cite{Kovalev-1}.
\end{abstract}

\maketitle

\section{Introduction}\label{sec-1}
Throughout this paper we study the strong solutions to the elliptic equation $L\varphi=0$. Let $\mathbb{C}$ be the complex plane, and let $\mathbb{D}(z_0, r)$ be a disk in $\mathbb{C}$, centered at $z_0$ and has the radius $r$, i.e.
$\mathbb{D}(z_0, r)=\{z\in \mathbb{C}: |z-z_0|< r\}$. For simplicity, the unit disk will be denoted by $\mathbb{D}$. Moreover, we use $\mathrm{Tr}(A)$ to denote the {\it trace} of $A$, and $A^T$ denotes the {\it transpose} of $A$.

\subsection{Linear elliptic differential equations}
Let $\Omega$ be a bounded domain in $\mathbb{R}^2$, which will be identified with the complex plane $\mathbb{C}$. Suppose $\varphi(x,y)\in C^2(\Omega)$ and $g(x,y)\in L^{\infty}(\Omega)$. Consider the following linear elliptic differential equations:
\begin{equation}\label{elliptic-1}L \varphi=a_{11}\varphi_{xx}+2a_{12}\varphi_{xy}+a_{22}\varphi_{yy}=g,\end{equation}
where the coefficients matrix $A=\big(a_{ij}(x,y)\big)_{2\times2}$ is a symmetric matrix, with measurable real coefficients.
For simplicity, let $B=D^2\varphi$ be the Hessian matrix of $\varphi$. Then (\ref{elliptic-1}) can be written as follows:
\begin{equation}\label{another-for-elp}L\varphi=\mathrm{Tr}(AB)=g.\end{equation}

Following the definition in \cite[Page 31]{DG}, the operator $L$ is said to be {\it elliptic} in $\Omega$, if
the coefficient matrix $A$ is positive; that is, if $\lambda$ and $\Lambda$ denote the minimum and maximum eigenvalues of $A$, then
\begin{equation}\label{uniform}0<\lambda |\xi|^2\leq \langle A\xi, \xi\rangle=\sum_{i,j=1}^{2}a_{ij}(\zeta)\xi_i\xi_j\leq\Lambda|\xi|^2\end{equation}
for all $\xi=(\xi_1, \xi_2)\in\mathbb{R}^2\setminus\{\mathbf{0}\}$. If $\lambda\geq\lambda_0>0$ for some constants $\lambda_0$, then $L$ is {\it strictly elliptic}.
If $\Lambda/\lambda$ is bounded in $\Omega$, then we call $L$ {\it uniformly elliptic} in $\Omega$.

The Dirichlet problem of the form: $L\varphi=g$, with $\varphi=0$ on the boundary $\partial\Omega$, arises in many applications from areas such as probability and stochastic processes.
It is well-known that, in the planar case, $a_{ij}\in L^{\infty}(\Omega)$ together with $L$ is uniformly elliptic in $\Omega$ can obtain well-posedness of the Dirichlet problem.
However, in higher dimensional case, the situation will be different and one should add the {\it Cod\'es condition} on the coefficients $a_{ij}$ to obtain the well-posedness of solutions. See for instance, the discussions in \cite{MPS}.

For the case $g=0$, the solutions will have useful properties in the theory of quasiregular mappings.

\subsection{\texorpdfstring{$K$-}\mbox{quasiregular mappings}}
Let $z=x+iy\in \Omega\subseteq\mathbb{C}$ and $f(z)=u(x,y)+iv(x,y)$ be a complex-valued function. The {\it formal derivatives} of $f$ are defined as follows:
$$f_{z}=\frac{1}{2}(f_x-if_y)\ \ \ \mbox{and}\ \ \ \ f_{\bar{z}}=\frac{1}{2}(f_x+if_y).$$

Following the definition in \cite{Astala}, given $K\geq1$, an orientation-preserving function $f\in  W^{1, 2}_{\mathrm{loc}}(\Omega; \mathbb{C})$ is said to be a $K$-quasiregular mapping (briefly $K$-q.r.) if and only if
\begin{equation}\label{k-qr} f_{\bar{z}}(z)=\mu(z) f_z(z),\ \ \mbox{for almost every}\ \  z\in\Omega, \end{equation}
where $\mu$, called the {\it Beltrami coefficient} of $f$, is a bounded measurable function satisfying
$$\|\mu\|_{\infty}\leq k=\frac{K-1}{K+1}<1.$$
If further, $f$ is a homeomorphic mapping, then $f$ is $K$-quasiconformal in $\Omega$.

The differential equation (\ref{k-qr}) is called the {\it Beltrami equation}. It is this equation that provides the connections from the geometric
theory of quasiconformal mappings to complex analysis and to elliptic PDEs.

Often, it is convenient to formulate (\ref{k-qr}) as the following distortion inequality
$$\frac{|f_z|+|f_{\bar{z}}|}{|f_z|-|f_{\bar{z}}|}=\frac{|D_f|^2}{J_f}\leq K<\infty,$$
where
$|D_f|=\max_{0\leq\alpha\leq\pi}|\partial_{\alpha}f|=|f_z|+|f_{\bar{z}}|$
is the maximal directional derivative of $f$, and $J_f=|f_z|^2-|f_{\bar{z}}|^2$ is the Jacobian of $f$.

For more details on the definition of higher-dimensional quasiconformal mappings and their applications, we refer to \cite{Vaisala, Matti} and \cite{Liu-Zhu, Liu-Petar-Zhu}.

\subsection{H\"older space} According to \cite[page 115]{Astala}, the H\"older space $\mathbf{C}^{\alpha}(\mathbb{C})$, $0<\alpha\leq1$, consists of continuous functions $f:\mathbb{C}\to\mathbb{C}$ that satisfy the H\"older condition:
$$\|f\|_{\mathbf{C}^{\alpha}(\mathbb{C})}=\sup\limits_{z\neq w}\frac{|f(z)-f(w)|}{|z-w|^{\alpha}}<\infty.$$

Let $f$ be a $K$-q.r. mapping of $\Omega$. By using the isoperimetric inequality, Green's formula, and Mori's theorem, one can prove that $f\in\mathbf{C}^{1/K}(\Omega)$ (see for example \cite[pages 80-82]{Astala}). Moreover, the H\"older exponent $1/K$ is best possible in the class of all $K$-quasiregular mappings, as the extremal function $|z|^{1/K-1}z$ shows.

\subsection{\texorpdfstring{$K$-}\mbox{quasiregular gradient mapping}}
Let $\varphi$ be a strong solution of the linear elliptic equation (\ref{elliptic-1}), with the operator $L$ uniformly elliptic in $\mathbb{D}$. Consider the complex-valued function
$f=\varphi_z$, i.e.
$$f=u+iv=\frac{\varphi_x-i\varphi_y}{2}.$$
In Section \ref{sec-21}, we will show that $f$ satisfies the following condition:
\begin{equation}\label{eq-11-28-1}|\nabla f|^2=2(|f_z|^2+|f_{\bar{z}}|^2)\leq 2KJ_f+K'\end{equation}
where
$$K=1+\frac{\Lambda}{\lambda}\ \ \ \mbox{and}\ \ \ K'=\frac{\|g\|_{\infty}^2}{2\lambda^2}.$$
Here $\lambda$ and $\Lambda$ are eigenvalues of the coefficient matrix $A$, and are given in (\ref{uniform}).

A function $f$ is a $(K, K')$-quasiregular mapping, if it satisfies (\ref{eq-11-28-1}). See for instance, \cite[Chapter 12]{DG}.
If in particular, $g=0$ and consequently $K'=0$, then such functions $f$ are $K$-quasiregular in $\Omega$.
It should be noted that in the homogeneous case $g=0$, after an orthogonal change of variables in the coefficient matrix $A$, the standard quasiregularity constant can be improved as (see Appendix \ref{apx1})
$$K=\frac{\Lambda}{\lambda}$$
which coincides with the assumption of \cite[(1.2)]{Kovalev-1}.

Notice that, $f_{\bar{z}}$ is real-valued, because
$f_{\bar{z}}=\varphi_{z\bar{z}}$
a.e. in $\Omega$.
We call $f\in  W^{1, 2}_{\mathrm{loc}}(\Omega; \mathbb{C})$ a $K$-{\it quasiregular gradient mapping}, if (\ref{k-qr}) holds and $f_{\bar{z}}$ is real.

\subsection{Improved H\"older exponent for \texorpdfstring{$K$-}\mbox{quasiregular gradient mapping}}
For a.e. $r\in(0, 1)$ the function $\theta\mapsto f(re^{i\theta})$ is absolutely continuous and its derivative is square integrable.
This allows us to expand it into the uniformly converging Fourier series
$$f(re^{i\theta})=\sum_{n=-\infty}^{\infty}c_{n}(r)e^{in\theta},$$
where
$$c_n(r)=\frac{1}{2\pi}\int_0^{2\pi}f(re^{i\theta})e^{-in\theta}\mathrm{d}\theta,\ \ \ n\in\mathbb{Z}.$$
It is convenient to use the polar coordinates to obtain
$$ f_z=\frac{e^{-i\theta}}{2}\left(f_r-\frac{i}{r}f_{\theta}\right)\ \ \ \mbox{and}\ \ \ f_{\bar{z}}=\frac{e^{i\theta}}{2}\left(f_r+\frac{i}{r}f_{\theta}\right).$$
Following the notation in \cite{Kovalev-1}, there are three real-valued functions $p$, $q$, $s$, so that
\begin{equation}\label{10-29-1}f_{\bar{z}}+e^{2i\theta}f_z=s-iq=\sum_{n=-\infty}^{\infty}c'_n(r)e^{i(n+1)\theta}\end{equation}
and
\begin{equation}\label{10-29-2}f_{\bar{z}}-e^{2i\theta}f_z=p+iq=-\sum_{n=-\infty}^{\infty}\frac{n}{r}c_n(r)e^{i(n+1)\theta}.\end{equation}
By using the arithmetic-geometric mean inequality $2ab\leq a^2+b^2$, Baernstein and Kovalev first improved the estimate of $J_f$ as follows (\cite[Lemma 2.1]{Kovalev-1})
$$J_f\geq \frac{1-k}{1+k}p^2+(1-k)q^2.$$
Later, applying Morrey's theorem and by comparing the Fourier series of $\int_0^{2\pi}p^2\mathrm{d}\theta$ and $\int_0^{2\pi}q^2\mathrm{d}\theta$ with $\int_0^{2\pi}J_f\mathrm{d}\theta$,
they finally obtain the improved H\"older exponent of $f$ as follows
\begin{equation}\label{kov-05}\alpha_{1}=\frac{(1-k)(\sqrt{k^2+16k+16}-k-2)}{2(1+k)}>\frac{1-k}{1+k}=\frac{1}{K}.\end{equation}

\subsection{Motivation} Although the H\"older exponent $\alpha_1$ in (\ref{kov-05}) is not optimal, `` it does provide the first H\"older regularity for $K$-quasiregular gradient mappings beyond the long-standing threshold $1/K$ ", as was said in \cite{Kovalev-1}. In this paper, we are interested in finding the optimal  H\"older exponent $\alpha$. We first estimate $J_f$ as follows:
\begin{align*}
  J_f & =\frac{1-|\mu|^2}{1+|\mu|^2}\left(q^2+\frac{p^2+s^2}{2}\right) \\
   & \geq \frac{(1-k^2-t)t}{(1-k^2)(1-t)}p^2+tq^2:=t_1p^2+tq^2,
\end{align*}
where $1-k<t<1-k^2$. Then, by Morrey's theorem, we show that the corresponding H\"older exponent is as follows:
\begin{equation}\label{eq-26-1-3-1}\alpha(t)=\frac{1}{2}\left(\sqrt{(t_1+t)^2+12t_1t}-(t_1+t)\right).\end{equation}
We prove in Appendix \ref{apx2} that $\alpha'(1-k)>0$, $\alpha'(1-k^2)<0$, and $\alpha''(t)<0$. Therefore, $\alpha(t)$ has a unique critical point $t_{\star}\in(1-k, 1-k^2)$  and the maximal value
$$\alpha(t_\star)=\max_{1-k<t<1-k^2}\alpha(t)$$
exists. Moreover, we prove that $t_\star$ is a solution of the following equation
\begin{align}\label{3-25-ev}\nonumber
  N_k(t)=&(16-16k^2+k^4)t^4-(64-80k^2+18k^4)t^3 \\
  & +(96-160k^2+69k^4-5k^6)t^2-(64-144k^2+96k^4-16k^6)t\\\nonumber
  &+16-48k^2+48k^4-16k^6=0.
\end{align}

In summary, our main results are as follows.

\begin{theorem}\label{thm1-2021-Oct-27}
Suppose that $\varphi\in W^{2, 2}_{\mathrm{loc}}(\Omega)$ is a strong solution of the uniformly elliptic equation $L \varphi=0$ in a domain $\Omega$.
Then
$\varphi\in\mathbf{C}^{1, \alpha}_{\mathrm{loc}}(\Omega)$, where
$$\alpha =\alpha(t_\star)$$
is given by $(\ref{eq-26-1-3-1})$, and $t_\star\in(1-k, 1-k^2)$ is a solution of the equation $(\ref{3-25-ev})$.
\end{theorem}

\begin{remark}
If in particular, choosing
$$t_0=(1-k)\left(1+\frac{1}{4}k\right),$$
then we have an explicit H\"older exponent $\alpha_0=\alpha(t_0)$ as follows
$$\alpha_0=\frac{1}{8}\left(\frac{(1-k)(4+k)\sqrt{144+k(4+k)(48+k(4+k))}}{(1+k)(3+k)}+k(3+k)-\frac{21+3k}{(3+k)(1+k)}-1\right).$$
It is not difficult to check that $\alpha(t_{\star})>\alpha_0>\alpha_1$, for any $0<k<1$. See Figure 1.
\end{remark}

\begin{corollary}\label{cor-2021-Oct-27}
Under the assumptions of Theorem $\ref{thm1-2021-Oct-27}$, we have
$\varphi\in\mathbf{C}^{1, \alpha_2}_{\mathrm{loc}}(\Omega)$, where
$$\alpha_2 =\frac{\sqrt{33}-3}{4}\cdot\frac{1-k^2}{1+k^2}>\frac{1-k}{1+k}, \ \ \mbox{if}\ \ k\in(0.2422, 1).$$
\end{corollary}

The rest of this paper is organized as follows: in Section \ref{sec-2}, we give the necessary terminologies, introduce some known results, and prove two lemmas which will be used in
proving our main results; in Section \ref{sec-3}, we prove Theorem \ref{thm1-2021-Oct-27} and Corollary \ref{cor-2021-Oct-27}; in Section \ref{apx}, we give some detailed calculations related to the existence of the unique critical point $t_{\star}$.

\begin{figure}
  \centering
  \includegraphics[width=0.8\textwidth]{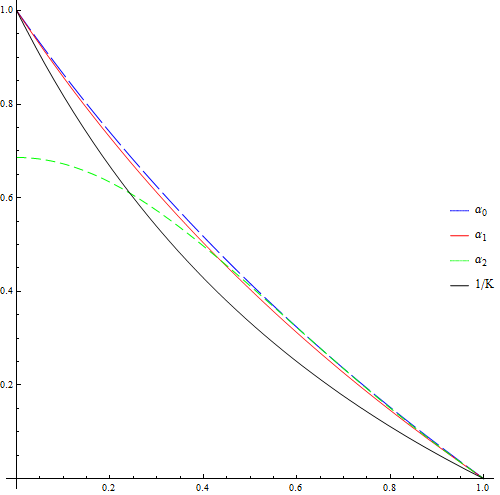}\\
  \caption{Comparison of H\"older exponent}\label{12}
\end{figure}

\section{Preliminaries}\label{sec-2}
In this section, we introduce some necessary terminology and prove two lemmas.
We start with the following Hessian matrix of $\varphi$.

Let $A=(a_{ij})_{2\times2}$ be the coefficient matrix of (\ref{elliptic-1}), and $B=D^2\varphi$ be the Hessian matrix of $\varphi$. Recall that we have assumed
$$f=u+iv=\varphi_z.$$
This implies $u=\varphi_x/2$ and $v=-\varphi_y/2$. Then one can write $B$ as follows:
\begin{equation*}       
B=(\xi_1, \xi_2)
=2\left(                 
  \begin{matrix}   
   u_x, & u_y \\  

    -v_x, & -v_y \\  
  \end{matrix}
\right)
\end{equation*}
where $u_y=-v_x$ and $\xi_j$ is the $j$-th column of $B$.

Elementary calculations show that
$$\mathrm{Tr}(ABB^{T})=\sum_{k=1}^2\langle A\xi_k, \xi_k\rangle$$
and of course, $B^T=B$ since $B$ is symmetric.
Moreover, we have the following equation which is not difficult to check
\begin{equation}\label{tr-det-eq1}\mathrm{Tr}(ABB^T)=\mathrm{Tr}(AB)\mathrm{Tr}(B^T)-\mathrm{Tr}(A)\mathrm{det}(B),\end{equation}
where $\mathrm{det}(B)$ is the determinant of $B$.

\subsection{\texorpdfstring{$f=\varphi_z$}\, \mbox{is a} \texorpdfstring{$K$-} quasiregular mapping}\label{sec-21}
It follows from the definition of $B$ that
\begin{equation*}       
\xi_1=2\left(                 
  \begin{matrix}   
   u_x \\  

    -v_x \\  
  \end{matrix}
\right)
\ \ \ \mbox{and}\ \ \
\xi_2=2\left(                 
\begin{matrix}   
   u_y \\  

    -v_y \\  
  \end{matrix}
\right).                
\end{equation*}
The uniform ellipticity of $L$ implies the following inequalities hold
\begin{equation}\label{eq-11-12-2}\lambda(|\xi_1|^2+|\xi_2|^2)\leq\mathrm{Tr}(ABB^T)=\langle A\xi_1, \xi_1\rangle+\langle A\xi_2, \xi_2\rangle\leq\Lambda(|\xi_1|^2+|\xi_2|^2)\end{equation}
where $\lambda$ and $\Lambda$ are positive and satisfy $\Lambda/\lambda\leq \mbox{const}$.

On the other hand, using (\ref{tr-det-eq1}) and (\ref{another-for-elp}) we have
\begin{align}\label{12-23-1}\nonumber
  \mathrm{Tr}(ABB^T) & =\mathrm{Tr}(B^T)\mathrm{Tr}(AB)-\mathrm{Tr}(A)\mathrm{det}(B) \\
  & =2(u_x-v_y)g-4(a_{11}+a_{22})(-u_xv_y+u_yv_x).
\end{align}
In order to estimate $\mathrm{Tr}(ABB^T)$, by setting $t=\lambda>0$, one can easily obtain
$$(u_x-v_y)g\leq\frac{t(u_x-v_y)^2+\frac{1}{t}g^2}{2}\leq\frac{2t(u_x^2+v_y^2)+\frac{1}{t}\|g\|_{\infty}^2}{2}.$$
This implies that
$$\mathrm{Tr}(B^T)\mathrm{Tr}(AB)\leq 2t(u_x^2+u_y^2+v_x^2+v_y^2)+\frac{1}{t}\|g\|_{\infty}^2.$$
Moreover, since $\lambda$, $\Lambda$ are the eigenvalues of $A$, we see that $\mathrm{Tr}(A)=a_{11}+a_{22}=\lambda+\Lambda>0$.
By using
$$J_f=u_xv_y-u_yv_x,$$
we have
$$\mathrm{Tr}(A)\mathrm{det}(B)=-4(\lambda+\Lambda)J_f.$$
Then
\begin{equation}\label{eq-11-28-3}\mathrm{Tr}(ABB^T)\leq2\lambda(u_x^2+u_y^2+v_x^2+v_y^2)+\frac{1}{\lambda}\|g\|_{\infty}^2+4(\lambda+\Lambda)J_f.\end{equation}

Recall that
$$|\xi_1|^2+|\xi_2|^2=\|B\|_{F}^2=4(u_x^2+u_y^2+v_x^2+v_y^2),$$
where $\|B\|_{F}$ is the Frobenius norm of $B$. We see from (\ref{eq-11-12-2}) and (\ref{eq-11-28-3}) that
$$2\lambda(u_x^2+u_y^2+v_x^2+v_y^2)\leq\frac{1}{\lambda}\|g\|_{\infty}^2+4(\lambda+\Lambda)J_f$$
or what is the same
$$|\nabla f|^2=2(|f_z|^2+|f_{\bar{z}}|^2)\leq\frac{1}{2\lambda^2}\|g\|_{\infty}^2+2\left(1+\frac{\Lambda}{\lambda}\right)J_f.$$
This shows that $f$ is a $(K, K')$-quasiregular mapping (see \cite[Chapter 12]{DG}).

If in particular $g=0$, then we see from (\ref{eq-11-12-2}) and (\ref{12-23-1}) that
\begin{equation}\label{eq-1-1130}\frac{\lambda}{\lambda+\Lambda}|\nabla f|^2\leq J_f\leq\frac{\Lambda}{\lambda+\Lambda}|\nabla f|^2\end{equation}
which shows that $f$ is a $K$-q.r. mapping with $K=1+\Lambda/\lambda$.

We remark here that, if $g=0$, then without using the arithmetic-geometric mean inequality $2ab\leq a^2+b^2$, one can directly deduce from (\ref{eq-11-12-2}) and (\ref{12-23-1}) to get a better estimate of $J_f$, that is (\ref{eq-1-1130}). Moreover, after an orthogonal change of variables, one can improve the quasiregularity constant as $K=\Lambda/\lambda$, see Appendix \ref{apx1}.

\subsection{Representation of \texorpdfstring{$J_f$}\, \mbox{in term of} \texorpdfstring{$p$, $q$,}\, \mbox{and} $s$}
For $z=re^{i\theta}\in\Omega$, following the notations in \cite{Kovalev-1}, we may assume that
\begin{equation}\label{sq} f_{\bar{z}}(z)+e^{2i\theta}f_z(z) =s(z)-iq(z) \end{equation}
and
\begin{equation}\label{pq}f_{\bar{z}}(z)-e^{2i\theta}f_z(z)=p(z)+iq(z),\end{equation}
where $p$, $q$, $s$ are real-valued functions.

We have the following lemma.
\begin{lemma}\label{lem-1}
Let $\varphi\in C^2(\Omega)$ be the strong solution of the lnear uniformly elliptic equation $L\varphi=0$, and $f=\varphi_z$. Then $f$ is $K$-q.r. in $\Omega$ and for a.e. $z\in\Omega$, we have
$$J_f=\frac{1-|\mu|^2}{1+|\mu|^2}\left(q^2+\frac{p^2+s^2}{2}\right),$$
where $\mu$ is the Beltrami coefficient of $f$.
\end{lemma}
\begin{proof}
It follows from (\ref{sq}) and (\ref{pq}) that
$$f_z=\frac{e^{-2i\theta}}{2}(s-p-2iq)\ \ \ \mbox{and}\ \ \ f_{\bar{z}}=\frac{1}{2}(s+p).$$
Since $g=0$, we see from (\ref{eq-1-1130}) that $f$ is a $K$-q.r. mapping. The Beltrami equation (\ref{k-qr}) gives the following equality
\begin{equation}\label{eq-2-11-30}(s+p)^2=|\mu(z)|^2((s-p)^2+4q^2),\end{equation}
where $\mu$ is the Beltrami coefficient satisfying $|\mu|\leq k<1$. Then
\begin{equation}\label{lem-1-1}ps=\frac{2|\mu|^2}{1+|\mu|^2}q^2-\frac{1-|\mu|^2}{2(1+|\mu|^2)}(s^2+p^2).\end{equation}

Notice that
\begin{equation}\label{lem-1-2}J_f=|f_z|^2-|f_{\bar{z}}|^2=-\mathrm{Re}\big((s+iq)(p+iq)\big)=q^2-ps.\end{equation}
Combining (\ref{lem-1-1}) and (\ref{lem-1-2}), we have
\begin{equation}\label{eq-3-10-30}J_f=\frac{1-|\mu|^2}{1+|\mu|^2}\left(q^2+\frac{p^2+s^2}{2}\right),\end{equation}
which completes the proof.
\end{proof}

Consider the function
$$J(r):=\int_{\mathbb{D}(0, r)}J_f(z)\mathrm{d}x\mathrm{d}y.$$
We are going to obtain the Fourier series of $J(r)$.

First, it follows from (\ref{10-29-1}) that
$$s=\frac{1}{2}\sum_{n=-\infty}^{\infty}\big(c_n'e^{i(n+1)\theta}+\bar{c}_n'e^{-i(n+1)\theta}\big)$$
and
$$q=-\frac{1}{2i}\sum_{n=-\infty}^{\infty}\big(c_n'e^{i(n+1)\theta}-\bar{c}_n'e^{-i(n+1)\theta}\big).$$
Similarly, from (\ref{10-29-2}), we obtain
$$p=-\frac{1}{2}\sum_{n=-\infty}^{\infty}\left(\frac{nc_n}{r}e^{i(n+1)\theta}+\frac{n\bar{c}_n}{r}e^{-i(n+1)\theta}\right)$$
and
$$q=-\frac{1}{2i}\sum_{n=-\infty}^{\infty}\left(\frac{nc_n}{r}e^{i(n+1)\theta}-\frac{n\bar{c}_n}{r}e^{-i(n+1)\theta}\right).$$

Second, by using Parseval's theorem, we have
$$\int_0^{2\pi}ps\mathrm{d}\theta=-\pi\sum_{n=-\infty}^{\infty}\mbox{Re}\left(\frac{n}{r}c_nc_{-n-2}'+\frac{n}{r}c_n\bar{c}_{n}'\right),$$
and
$$\int_0^{2\pi}q^2\mathrm{d}\theta=-\pi\sum_{n=-\infty}^{\infty}\mbox{Re}\left(\frac{n}{r}c_nc_{-n-2}'-\frac{n}{r}c_n\bar{c}_n'\right).$$
Therefore, since $\frac{\mathrm{d}}{\mathrm{d}\rho} |c_n|^2=c_n'\bar{c_n}+\bar{c}_n'c_n$, one has (\cite[(2.8)]{Kovalev-1})
\begin{align}\label{24-614-1}
  J(r) & =\int_{\mathbb{D}(0, r)}J_f(z)\mathrm{d}x\mathrm{d}y=\int_0^r\rho\mathrm{d}\rho\int_0^{2\pi}(q^2-ps) \mathrm{d}\theta \\\nonumber
   & =\int_0^r\pi\sum_{n=-\infty}^{\infty}n\bigg(c_n\bar{c_n}'+\bar{c_n}c_n'\bigg)\mathrm{d}\rho\\\nonumber
   &=\pi\sum_{n=-\infty}^{\infty}n|c_n(r)|^2.
\end{align}

\subsection{Fourier series of the integration of functions $p^2$, $q^2$, and $s^2$}

Applying Parseval's theorem again, one can easily obtain
$$\int_0^{2\pi}s^2\mathrm{d}\theta=\pi\sum_{n=-\infty}^{\infty}|c_n'|^2+\pi\sum_{n=-\infty}^{\infty}\mbox{Re}(c_n'c_{-n-2}')=\frac{\pi}{2}\sum_{n=-\infty}^{\infty}|c_n'+\bar{c}_{-n-2}'|^2,$$
and
$$\int_0^{2\pi}q^2\mathrm{d}\theta=\pi\sum_{n=-\infty}^{\infty}|c_n'|^2-\pi\sum_{n=-\infty}^{\infty}\mbox{Re}(c_n'c_{-n-2}')=\frac{\pi}{2}\sum_{n=-\infty}^{\infty}|c_n'-\bar{c}_{-n-2}'|^2.$$
On the other hand, by using $d_n=nc_n$, we have (\cite[(2.9) and (2.10)]{Kovalev-1})
\begin{align}\label{p2-123}
  \int_0^{2\pi}p^2\mathrm{d}\theta& =\pi\sum_{n=-\infty}^{\infty}\left(\frac{n}{r}\right)^2|c_n|^2+\pi\sum_{n=-\infty}^{\infty}\frac{n(n+2)}{r^2}\mbox{Re}(c_nc_{-n-2}) \\\nonumber
   & =\frac{\pi}{2r^2}\sum_{n=-\infty}^{\infty}|d_{n-1}+\bar{d}_{-n-1}|^2
\end{align}
and
\begin{align}\label{q2-123}
  \int_0^{2\pi}q^2\mathrm{d}\theta & =\pi\sum_{n=-\infty}^{\infty}\left(\frac{n}{r}\right)^2|c_n|^2-\pi\sum_{n=-\infty}^{\infty}\frac{n(n+2)}{r^2}\mbox{Re}(c_nc_{-n-2}) \\\nonumber
  & =\frac{\pi}{2r^2}\sum_{n=-\infty}^{\infty}|d_{n-1}-\bar{d}_{-n-1}|^2.
\end{align}
Now, summing all the integrals from above, one has
\begin{equation}\label{main-ineq}\int_0^{2\pi}\left(q^2+\frac{p^2+s^2}{2}\right)\mathrm{d}\theta=\pi\sum_{n=-\infty}^{\infty}\left(\frac{n}{r}\right)^2|c_n|^2+\pi\sum_{n=-\infty}^{\infty}|c_n'|^2.\end{equation}

Moreover, by letting $c_n(r)=A_n(r)+iB_n(r)$ and $c_n'(r)=A_n'(r)+iB_n'(r)$, where $A_n$, $A_n'$, $B_n$, $B_n'$ are real-valued, we conclude from (\ref{10-29-1}) and (\ref{10-29-2}) that
\begin{align}\label{11-10-1}
  q & =-\sum_{n=-\infty}^{\infty}\bigg(A_n'(r)\sin(n+1)\theta+B_n'(r)\cos(n+1)\theta\bigg) \\\nonumber
   & =-\sum_{n=-\infty}^{\infty}\frac{n}{r}\bigg(A_n(r)\sin(n+1)\theta+B_n(r)\cos(n+1)\theta\bigg).
\end{align}
Then
$$\frac{1}{2\pi}\int_0^{2\pi}q(re^{i\theta})\mathrm{d}\theta=-B_{-1}'(r)=\frac{B_{-1}(r)}{r}.$$
Next, for $k=1, 2, \cdots$, we see from (\ref{11-10-1}) that
\begin{align*}
  \int_0^{2\pi}q(re^{i\theta}) e^{ik\theta}\mathrm{d}\theta &= \pi i(c_{-k-1}'-\bar{c}_{k-1}') \\
   & =\pi i\left(\frac{-k-1}{r}c_{-k-1}-\frac{k-1}{r}\bar{c}_{k-1}\right).
\end{align*}
Or equivalently,
\begin{equation}\label{11-10-2}c_{-n-2}'-\bar{c}_{n}'=\frac{-n-2}{r}c_{-n-2}-\frac{n}{r}\bar{c}_{n},\ \ \ \mbox{for}\ \ \  n=0, 1, 2, \cdots.\end{equation}

Let $$J'(r)=\frac{\mathrm{d}}{\mathrm{d}r}J(r)=r\int_0^{2\pi}J_f(re^{i\theta})\mathrm{d}\theta.$$
Since $f$ is a $K$-q.r. mapping, Morrey's lemma says that if one can find a constant $0<\alpha\leq1$, such that
$$J'(r)\geq\frac{2\alpha}{r}J(r)$$
then $f$ is H\"older continuous with the H\"older exponent equals to $\alpha$. See \cite[Page 301]{Kovalev-1} and \cite{Astala-JMPA, Morrey-TAMS} for more details.
Now, the Fourier series of $J(r)$ is already given by (\ref{24-614-1}), and one can also find the Fourier series of $J'(r)$ by using (\ref{eq-3-10-30}) and (\ref{main-ineq}).
However, since the function $s^2$ contains the coefficients $c_n'$, it is not easy to compare $J'(r)$ with $J(r)$ directly applying (\ref{11-10-2}).
One possible way is to estimate the integration of $s^2$ in terms of the functions $p^2$ and $q^2$.
\subsection{Representing the functions $s$ and $q$ in terms of $p$}
Let $\mu=a+ib=|\mu|e^{i\xi}.$
It follows from (\ref{10-29-1}) and (\ref{10-29-2}) that
$$f_{\bar{z}}=\frac{s+p}{2},\ \ \ f_z=e^{-2i\theta}\left(\frac{s-p}{2}-iq\right).$$
By using $f_{\bar{z}}=\mu f_z$, we have
$$\frac{s+p}{2}=|\mu|\left(\frac{s-p}{2}-iq\right)e^{i(\xi-2\theta)}.$$
Solving this equation and setting $\varsigma=\xi-2\theta$, one has
\begin{equation}\label{spq}\left\{
\begin{aligned}
	s&=&\frac{|\mu|+\cos\varsigma}{|\mu|-\cos\varsigma}p;\\
	q&=&\frac{\sin\varsigma}{\cos\varsigma-|\mu|}p.
\end{aligned}
\right.\end{equation}

We have the following lemma.

\begin{lemma}\label{lem-2}
Let $z=re^{i\theta}\in\mathbb{D}$, and let $p(z)$, $q(z)$, $s(z)$ be real-valued functions which are given in $(\ref{sq})$ and $(\ref{pq})$. Then
$$J_f\geq\frac{(1-k^2-t)t}{(1-k^2)(1-t)}p^2+tq^2,\ \ \ 1-k<t<1-k^2.$$
\end{lemma}
\begin{proof}
Recall that $J_f=q^2-ps$. It follows from (\ref{spq}) that
$$J_f=\frac{1-|\mu|^2}{(|\mu|-\cos\varsigma)^2}p^2.$$
In order to estimate $J_f$ in the form of
\begin{equation}\label{jf-eqtimate}J_f\geq t_1p^2+t_2q^2,\end{equation}
it is equivalent to show the following:
\begin{align*}
  1-|\mu|^2 & \geq t_1(|\mu|-\cos\varsigma)^2+t_2\sin^2\varsigma\\
   & =t_1|\mu|^2-2t_1|\mu|\cos\varsigma+(t_1-t_2)\cos^2\varsigma+t_2,
\end{align*}
or what is the same
$$(t_2-t_1)\cos^2\varsigma+2t_1|\mu|\cos\varsigma+1-t_2-(1+t_1)|\mu|^2\geq0.$$
Choosing $0<t_1<t_2<1$. The above inequality holds if its discriminant is non-positive, i.e.
$$\Delta=(2t_1|\mu|)^2-4(t_2-t_1)(1-t_2-(1+t_1)|\mu|^2)\leq0.$$
Simplify the above inequality, we have
$$\Delta=(t_2-t_1)(t_2-1+|\mu|^2)+t_1t_2|\mu|^2\leq(t_2-t_1)(t_2-1+k^2)+t_1t_2k^2\leq0,$$
since $|\mu|\leq k<1$.

Now, we may collect $t_1$ and rewrite the above inequality as follows
$$(1-k^2)(1-t_2)t_1-(1-k^2-t_2)t_2\leq0,\ \ \ 0<t_1<t_2\leq1.$$
Or what is the same
\begin{equation}\label{parameter-13}t_1\leq\frac{(1-t_2-k^2)t_2}{(1-t_2)(1-k^2)}.\end{equation}
This shows that in order to prove (\ref{jf-eqtimate}), we may need to choose $0<t_1<t_2<1$, such that (\ref{parameter-13}) holds.

It is easy to check that
$$\frac{(1-t_2-k^2)t_2}{(1-t_2)(1-k^2)}<t_2.$$
Therefore, if we choose
$$t_1=\frac{(1-t-k^2)t}{(1-t)(1-k^2)},\ \ \ t_2=t,$$
then
$$J_f\geq t_1p^2+t_2q^2$$
holds.

Moreover, since $t_1>0$, we have $t<1-k^2$. On the other hand, recall that in \cite{Kovalev-1}, the authors choose $t=1-k$, and thus, $t_1(1-k)=(1-k)/(1+k)$.
In the rest of this paper, in order to get a better exponent, we may assume that $t>1-k$.
\end{proof}

\begin{remark}
Notice that
$$\frac{\mathrm{d}t_1}{\mathrm{d} t}=\frac{(1-t)^2-k^2}{(1-k^2)(1-t)^2}.$$
The critical point is $t=1-k$, and in this case, one has $t_1=(1-k)/(1+k)$.

Moreover, since
$$\frac{\mathrm{d}^2 t_1}{\mathrm{d} t^2}\bigg|_{t=1-k}=\frac{-2k^2}{(1-k^2)(1-t_2)^3}<0.$$
We see that the following estimate
$$J_f\geq \frac{1-k}{1+k} p^2+(1-k)q^2$$
is in fact, the optimal estimate for $J_f$.

However, what we want is not the estimate of $J_f$, but the largest H\"older exponent
$$\alpha=\frac{1}{2}\left(\sqrt{(t_1+t_2)^2+12t_1t_2}-(t_1+t_2)\right).$$

\end{remark}

\section{Proofs of the main results}\label{sec-3}

\subsection{Proof of Theorem \ref{thm1-2021-Oct-27}}\label{sec-31}
The proof mainly follows from \cite{Kovalev-1}. However, for the reader's convenience and for the completeness of the proof, we add some details. Let $z_0\in\Omega$ and $0<R<\mbox{dist}(z_0, \partial\Omega)$.
Our goal is to find the optimal constant $0<\alpha\leq1$ such that
$$\int_{\mathbb{D}(z_0, r)}|Df(z)|^2\mathrm{d}x\mathrm{d}y\leq K\left(\frac{r}{R}\right)^{2\alpha}\int_{\mathbb{D}(z_0, R)}|Df(z)|^2\mathrm{d}x\mathrm{d}y,$$
where $K>1$ is a fixed constant and $0\leq r\leq R$.

Without loss of generality, one can assume that $z_0=0$ and $R=1$. Since $f$ is $K$-q.r. in $\Omega$, we have $1/K|Df|^2\leq J_f\leq|Df|^2$. Thus, one can reduce the above inequality as
\begin{equation}\label{thm-1-1}\int_{\mathbb{D}(0, r)}J_f(z) \mathrm{d}m(z)\leq r^{2\alpha}\int_{\mathbb{D}(0, 1)}J_f(z)\mathrm{d}x\mathrm{d}y,\end{equation}
where $0\leq r\leq1$.
Once this is done, Morrey's lemma will imply $f\in\mathbf{C}^{\alpha}(\Omega)$.

Let
$$J(r)=\int_{\mathbb{D}(0, r)}J_f(z) \mathrm{d}x\mathrm{d}y.$$
It is easy to see that $J(r)$ is an increasing absolutely continuous function on $[0, 1]$. Following \cite[(2.16)]{Kovalev-1},
to prove (\ref{thm-1-1}), it is sufficient for us to show the inequality
$$J'(r)=r\int_0^{2\pi}J_f(re^{i\theta})\mathrm{d}\theta\geq \frac{2\alpha}{r}J(r).$$

It follows from Lemma \ref{lem-2} that
$$J_f\geq\frac{(1-k^2-t)t}{(1-k^2)(1-t)}p^2+tq^2,\ \ \ 1-k<t<1-k^2.$$
For simplicity, let
$$t_1=\frac{(1-k^2-t)t}{(1-k^2)(1-t)}\ \ \ \mbox{and}\ \ \ t_2=t.$$
Then $t_2\geq t_1$. By using (\ref{p2-123}) and (\ref{q2-123}), one has
\begin{align*}
  J'(r) & \geq rt_1\int_0^{2\pi}p^2\mathrm{d}\theta+rt_2\int_0^{2\pi}q^2\mathrm{d}\theta \\
   & \geq\frac{\pi}{r}\sum_{n=2}^{\infty}\left(t_1|d_{n-1}+\bar{d}_{-n-1}|^2+t_2|d_{n-1}-\bar{d}_{-n-1}|^2\right).
\end{align*}
On the other hand, (\ref{24-614-1}) implies
\begin{equation}\label{Jr}J(r)=\pi\sum_{n=-\infty}^{\infty}n|c_n|^2=\pi\sum_{n\neq0}\frac{|d_n|^2}{n}\leq\pi\sum_{n=2}^{\infty}\left(\frac{|d_{n-1}|^2}{n-1}-\frac{|d_{-n-1}|^2}{n+1}\right).\end{equation}

Now, in order to compare $J'(r)$ and $J(r)$, by using (\ref{Jr}), we must find the largest possible constant $C$ such that the inequality
\begin{equation}\label{thm-eq3-123}
t_1|d_{n-1}+\bar{d}_{-n-1}|^2+t_2|d_{n-1}-\bar{d}_{-n-1}|^2\geq C\left(\frac{|d_{n-1}|^2}{n-1}-\frac{|d_{-n-1}|^2}{n+1}\right),\end{equation}
holds for all $n\geq2$. Once this is done, then the corresponding H\"older exponent of this case is $\alpha_2=C/2$.

Fix $t>0$ and let $\zeta=\zeta_1+i\zeta_2=\bar{d}_{-n-1}/d_{n-1}\in\mathbb{C}$. Following the proof of \cite{Kovalev-1}, and notice that the items containing $\zeta_2=\mbox{Im}\zeta$ are all positive on the left-hand side of (\ref{thm-eq3-123}), while negative on the other side. One can reduce (\ref{thm-eq3-123}) in the form by setting $d_{n-1}=1$ and $\bar{d}_{-n-1}=\zeta=\zeta_1\in\mathbb{R}$. That is
$$t_1(1+\zeta)^2+t_2(1-\zeta)^2\geq C\left(\frac{1}{n-1}-\frac{\zeta^2}{n+1}\right),$$
or what is the same
\begin{equation}\label{thm-eq-4}\left(t_1+t_2+\frac{C}{n+1}\right)\zeta^2+2(t_1-t_2)\zeta +\left(t_1+t_2-\frac{C}{n-1}\right)\geq 0.
\end{equation}

For
$$1-k<t< 1-k^2$$
we have
\begin{equation}\label{plus}t_1+t_2=\frac{t(2-2t-k^2(2-t))}{(1-k^2)(1-t)}>0\end{equation}
and
\begin{equation}\label{multip}t_1t_2=\frac{(1-k^2-t)t^2}{(1-k^2)(1-t)}>0.\end{equation}
Therefore, the first coefficient of (\ref{thm-eq-4}) satisfies
$$0<t_1+t_2+\frac{C}{n+1}\leq t_1+t_2+\frac{C}{3}.$$
To find the best $C$, we only need to consider the case $n=2$ in more detail. In fact, the inequality (\ref{thm-eq-4})  holds if its discriminant is non-positive, i.e.
$$4(t_1-t_2)^2-4\left(t_1+t_2+\frac{C}{3}\right)\left(t_1+t_2-C\right)\leq0.$$
This is equivalent to
$$\left(C+t_1+t_2\right)^2\leq(t_1+t_2)^2+12t_1t_2.$$
Then
$$C\leq\sqrt{(t_1+t_2)^2+12t_1t_2}-(t_1+t_2).$$
Therefore, we can choose
\begin{equation}\label{3-25}\alpha(t)=\frac{1}{2}\left(\sqrt{(t_1+t_2)^2+12t_1t_2}-(t_1+t_2)\right)\end{equation}
and (\ref{eq-26-1-3-1}) holds.

It should be noted that for $1-k<t<1-k^2$, we have (see Appendix \ref{apx2})
$$\alpha'(1-k)>0,\ \ \ \alpha'(1-k^2)<0,\ \ \  \mbox{and}\ \ \  \alpha''(t)<0.$$
This shows that $\alpha(t)$ has a unique critical point $t_\star\in(1-k, 1-k^2)$.

Let
$$S=t_1+t_2\ \ \ \mbox{and}\ \ \ P=t_1t_2$$
be given by (\ref{plus}) and (\ref{multip}).
Now, we find the critical point $t_\star$ as follows.
It follows from (\ref{3-25}) that
$$2\alpha'=\frac{SS'+6P'}{\sqrt{S^2+12P}}-S'=0.$$
Squaring the above equation, we obtain
\begin{equation}\label{3-25-1}SS'P'+3P'-PS'^2=0.\end{equation}

Elementary calculations and using (\ref{plus}), (\ref{multip}), show that
\begin{equation}\label{3-25-2}S'=\frac{2(1-t)^2-k^2(1+(1-t)^2)}{(1-k^2)(1-t)^2}\ \ \ \mbox{and}\ \ \ P'=\frac{2(1-t)^2-k^2(2-t)}{(1-k^2)(1-t)^2}t.\end{equation}
Substituting (\ref{3-25-2}) into (\ref{3-25-1}) and taking the numerator part, we see that the critical point $t=t_\star$ is one of the solutions to the following equation
\begin{align*}\label{3-25-e}\nonumber
  N_k(t)=&(16-16k^2+k^4)t^4-(64-80k^2+18k^4)t^3 \\
  & +(96-160k^2+69k^4-5k^6)t^2-(64-144k^2+96k^4-16k^6)t\\\nonumber
  &+16-48k^2+48k^4-16k^6=0.
\end{align*}

This completes the proof. \qed

\subsection{Proof of Corollary \ref{cor-2021-Oct-27}} It follows from Lemma \ref{lem-1} that
\begin{align*}
  J_f & =\frac{1-|\mu|^2}{1+|\mu|^2}\left(q^2+\frac{p^2+s^2}{2}\right) \\
   & \geq\frac{1-k^2}{1+k^2}\left(q^2+\frac{p^2}{2}\right).
\end{align*}
According to the proof of Theorem \ref{thm1-2021-Oct-27}, for
$$t_1=\frac{1}{2}\cdot\frac{1-k^2}{1+k^2}, \ \ \ t_2=\frac{1-k^2}{1+k^2},$$
we have
$$\alpha=\frac{1}{2}\left(\sqrt{(t_1+t_2)^2+12t_1t_2}-(t_1+t_2)\right)=\frac{\sqrt{33}-3}{4}\cdot\frac{1-k^2}{1+k^2}.$$

This completes the proof. \qed

\section{Appendix}\label{apx}
\subsection{The standard quasiregularity constant $K$}\label{apx1}
In the homogeneous case $g=0$, after an orthogonal change of variables on may assume
\begin{equation*}       
A=\left(                 
  \begin{matrix}   
   \Lambda, & 0 \\  

    0, & \lambda \\  
  \end{matrix}
\right),
\ \ \ m:=\frac{\Lambda}{\lambda}>1.
\end{equation*}
Writing
\begin{equation*}       
D^2\varphi=\left(                 
  \begin{matrix}   
   a, & b \\  

    c, & d \\  
  \end{matrix}
\right),
\end{equation*}
the equation $L\varphi=0$ gives
$$\Lambda a+\lambda c=0,\ \ \ c=-ma.$$
For
$$f=\varphi_z=\frac{\varphi_x-i\varphi_y}{2},$$
one has
$$f_z=\frac{(1+m)a-2ib}{4}\ \ \ \mbox{and}\ \ \ f_{\bar{z}}=\frac{(1-m)a}{4}.$$
Hence
$$|\mu|=\left|\frac{f_{\bar{z}}}{f_z}\right|\leq\frac{\Lambda-\lambda}{\Lambda+\lambda}.$$
Therefore, the Beltrami parameter should be taken as
$$k=\frac{\Lambda-\lambda}{\Lambda+\lambda}$$
and the standard quasiregularity constant is
$$K=\frac{1+k}{1-k}=\frac{\Lambda}{\lambda}.$$
It should be noted that in \cite[(1.2)]{Kovalev-1}, Baernstein and Kovalev normalized with $\Lambda=\sqrt{K}$ and $\lambda=1/\sqrt{K}$.

\subsection{$\alpha(t)$ has a unique critical point in the interval $(1-k, 1-k^2)$}\label{apx2}
Recall that
$$2\alpha'=\frac{SS'+6P'}{\sqrt{S^2+12P}}-S'.$$
Since $t_1(1-k)=(1-k)/(1+k)$, $t_1(1-k^2)=0$, $t_1'(1-k)=0$, and $t_1'(1-k^2)=-1/k^2$, we have
$$\alpha'(1-k)=\frac{8-k^2-7k-2\frac{1-k}{1+k}\sqrt{16+16k+k^2}}{2(1-k)\sqrt{16+16k+k^2}}>0$$
and
$$\alpha'(1-k^2)=-\frac{6}{k^2}<0.$$

On the other hand, elementary calculations show that
\begin{align*}
  2\alpha'' & =\frac{S'^2+6P''+SS''}{(S^2+12P)^{1/2}}-\frac{(SS'+6P')^2}{(S^2+12P)^{3/2}}-S'' \\
   & \leq \frac{S'^2+6P''}{(S^2+12P)^{1/2}}-\frac{(SS'+6P')^2}{(S^2+12P)^{3/2}}.
\end{align*}
To show $\alpha''<0$, it is sufficient to show that
$$(S'^2+6P'')(S^2+12P)-(SS'+6P')^2\leq0$$
i.e.
$$2PS'^2+P''S^2+12PP''-2P'S'S-6P'^2\leq0.$$
Now, applying (\ref{plus}), (\ref{multip}), and (\ref{3-25-2}), the above inequality is equivalent to
$$-\frac{2k^2t^3[(16-20k^2+5k^4)t^2-(32-52k^2+20k^4)t+16(1-k^2)^2]}{(1-k^2)^3(1-t)^5}\leq0.$$

We may only consider the numerator part. For this, let
$$a=16-20k^2+5k^4,\ \ \ b=32-52k^2+20k^4,\ \ \ c=16(1-k^2)^2.$$
Consider the function
$$\varphi(t)=at^2-bt+c.$$
Once we prove that $\varphi(t)>0$ for $1-k<t<1-k^2$, then $\alpha''(t)<0$, and this is what we need.

But this is easy to check, because the coefficients $a$, $b$, and $c$ are positive, since
$$a=5(1-k^2)^2+10(1-k^2)+1>0\ \ \ \mbox{and}\ \  b=20(1-k^2)^2+12(1-k^2)>0.$$
The discriminant
$$\Delta=b^2-4ac=80k^4(1-k^2)^2>0.$$
There are two real roots of $\varphi(t)$, which are given as follows
$$x_{1, 2}=\frac{b\mp\sqrt{\Delta}}{2a}=\frac{20(1-k^2)^2+12(1-k^2)\mp\sqrt{80}k^2(1-k^2)}{10(1-k^2)^2+20(1-k^2)+2}.$$
Notice that
$$20(1-k^2)+12-\sqrt{80}k^2>10(1-k^2)^2+20(1-k^2)+2$$
i.e.
$$20-\sqrt{80}-10k^2\geq0$$
holds. We have the smallest root
$$x_1>1-k^2$$
and thus, $\varphi(t)>0$ for $1-k<t<1-k^2$.

\vspace*{5mm}
\noindent {\bf Acknowledgments}.
The authors would like to thank Professor Leonid Kovalev and Professor David Kalaj for their helpful comments and suggestions on this paper, and thank Professor Giuseppe Di Fazio for his kindly help on the calculations of the critical point of $\alpha(t)$.

\vspace*{5mm}
\noindent{\bf Funding.}
The research of the authors were supported by NSF of China (No. 12271189), NSF of Guangdong Province (Grant No. 2024A1515010467, 2026A1515012333), STU Scientific Research Initiation Grant NTF25017T, and Fujian Alliance of Mathematics (Grant No. 2023SXLMMS07).

\vspace*{5mm}
\noindent {\bf Conflict of Interests}.
The authors declare that there is no conflict of interests regarding the publication of this paper.

\vspace*{5mm}
\noindent {\bf Data Availability Statement}.
The authors confirm that the data supporting the findings of this study are available within the article and its supplementary materials.

\end{document}